\documentclass[12pt]{amsart}

\usepackage{amssymb}
\usepackage{eucal}
\usepackage{color}

\DeclareSymbolFont{SY}{U}{psy}{m}{n}
\DeclareMathSymbol{\emptyset}{\mathord}{SY}{'306}

\theoremstyle{plain}
\newtheorem{thm}{Theorem}
\newtheorem{lem}[thm]{Lemma}

\newtheorem{pro}[thm]{Proposition}
\newtheorem{cor}[thm]{Corollary}

\theoremstyle{definition}
\newtheorem{defn}[thm]{Definition}

\def\beq{\begin{eqnarray}}
\def\eeq{\end{eqnarray}}
\def\beqa{\begin{eqnarray*}}
\def\eeqa{\end{eqnarray*}}

\begin{document}
\title[Generalized Bundle Shift]{ Generalized Bundle Shift with Application to Multiplication operator on the Bergman space}

\author[Douglas]{Ronald G. Douglas}
\address[Douglas]{Department of Mathematics, Texas A\&M University,
College Station, Texas 77843, USA} \email{rgd@math.tamu.edu}
\author[Keshari]{Dinesh Kumar Keshari}
\address[Keshari]{Department of Mathematics, Texas A\&M University,
College Station, TX 77843, USA} \email{kesharideepak@gmail.com}
\author[Xu]{Anjian Xu}
\address[Xu]{Department of Mathematics, Chongqing University of Technology and \\
Department of Mathematics, Texas A\&M University, College Station, TX 77843, USA} \email{xuaj@cqut.edu.cn}

\begin{abstract}
Following upon results of Putinar, Sun, Wang, Zheng and the first author, we provide models for the restrictions of the multiplication by a finite Balschke product on the Bergman space in the unit disc to its reducing subspaces. The models involve a generalization of the notion of bundle shift on the Hardy space introduced by Abrahamse and the first author to the Bergman space. We develop generalized bundle shifts on more general domains. While the characterization of the bundle shift is rather explicit, we have not been able to obtain all the earlier results appeared, in particular, the facts that the number of the minimal reducing subspaces equals the number of connected components of the Riemann surface $B(z)=B(w)$ and the algebra of commutant of $T_{B}$ is commutative, are not proved. Moreover, the role of the Riemann surface is not made clear also.
\end{abstract}
\begin{subjclass}
\subjclass[2010]{47B32, 47B35}
\end{subjclass}
\maketitle
\keywords{Bergman Bundle Shift, Bergman space, Finite Blaschke product, Reducing subspace}
\section{Introduction}
In the study of bounded linear operators, we can ask about the
reducing subspaces of a general bounded operator defined on a
separable Hilbert space. But, in general, we can not say much about
reducing subspace of arbitrary bounded linear operators. If we
restrict attention to some special class of operators, then we can
get more information about their reducing subspaces. One set of examples consist of multiplication operators by
finite Blachke products on the Bergman space on the unit disc $\mathbb{D}$.


Let $\mathcal{O}(\mathbb{D})$ be the set of holomorphic functions
on $\mathbb{D}$ and $L_a^2(\mathbb{D})$ be the
Bergman space of functions in $\mathcal{O}(\mathbb{D})$ satisfying
\[\|f\|^{2}=-\frac{1}{2\pi i}\int_{\mathbb{D}}|f(z)|^2\\dz\wedge d\bar{z}<\infty\]

Let $T$ be a bounded linear operator on
$L_a^2(\mathbb{D})$, a subspace $\mathfrak{M}$ of
$L_a^2(\mathbb{D})$ is called a \emph{reducing subspace} of $T$ if
$T(\mathfrak{M})\subseteq \mathfrak{M}$ and
$T^*(\mathfrak{M})\subseteq \mathfrak{M}$. A reducing subspace
$\mathfrak{M}$ of $T$ is called \emph{minimal} if for every reducing
subspace $\tilde{\mathfrak{M}}$  of $T$ such that
$\tilde{\mathfrak{M}}\subseteq \mathfrak{M}$ then either
$\tilde{\mathfrak{M}}= \mathfrak{M}$ or $\tilde{\mathfrak{M}}=0$.

An $n$-th order Blaschke product $B$ is the analytic function on
$\mathbb{D}$ given by
$$B(z)=e^{i\theta}\prod_{i=1}^{n}\frac{z-a_i}{1-\bar{a}_iz},$$
where $\theta$ is a real number and $a_i\in\mathbb{D}$ for $1\leq
i\leq n$.

Let $T_B$ be the multiplication operator on $L_a^2(\mathbb{D})$
by $B$. Zhu first studied the reducing subspaces of $T_B$, and showed that $T_B$ has exactly two distinct minimal reducing subspaces when the Blachke product is of order $2$
(cf. \cite{zhu}). Motivated by this fact, Zhu conjectured that the number of minimal reducing subspaces of
$T_B$ equals to the order of $B$(cf. \cite{zhu}). Guo, Sun, Zheng and Zhong showed that in general this is
not true  (cf. \cite{gszz}), and they found that the number of minimal reducing subspaces of
$T_B$ equals to the number of connected components of the Riemann surface of $B(z)=B(w)$ when the order of $B$ is 3,4,6. Then they conjectured that the number of minimal reducing subspaces of $T_B$ equals to the number of connected components of the Riemann surface of $B(z)=B(w)$ for any finite Blaschke product (the refined Zhu's conjecture, cf. \cite{gszz}).

Set $B(z):=e
^{i\theta}\prod\limits_{i=1}^{n}\frac{z-a_i}{1-\bar{a}_i z}=\tfrac{P(z)}{Q(z)}$ where $P(z)=e
^{i\theta}\prod\limits_{i=1}^{n}(z-a_i)$ and
$Q(z)=\prod\limits_{i=1}^{n}(1-\bar{a}_i z)$, and let $(\mathcal{W}^*(T_B))^{\prime}$ denote the \emph{double commutant} of $T_B$, i.e., all $T$ in $\mathcal{L}(L_a^2(\mathbb{D}))$ such that $TT_B=T_BT$ and $TT_B^*=T_B^*T$. It is easy to see $(\mathcal{W}^*(T_B))^{\prime}$ is a Von Neumann algebra.
The first author, Sun and Zheng proved following theorem in \cite{dsz}.

\begin{thm}
If $B$ is a finite Blaschke product, then $(\mathcal{W}^*(T_B))^{\prime}$ has dimension $q$, where $q$ is the number of connected components of the Riemann surface $S_f$, where $f(z,w)=P(w)Q(z)-P(z)Q(w)$.
\end{thm}

Zhu's refined conjecture is then proved once one can prove $(\mathcal{W}^*(T_B))^{\prime}$ is commutative, in the same paper, they proved affirmatively for the cases $n\leq8$. Putinar, Wang and the first author \cite{dpw} proved the following theorem for the general cases by using the notion of local inverses of a finite Balschke product which is introduced by Thompson \cite{t}.

\begin{thm}
$(\mathcal{W}^*(T_B))^{\prime}$ is commutative for any finite Blaschke product.
\end{thm}

One can show that the restrictions of $T_{B}$ onto its minimal reducing subspaces are not unitary equivalent to each other. Then the following question comes naturally: what are these operators obtained by restricting $T_{B}$ onto its minimal reducing subspaces?

Firstly let us consider the special case when $B(z)=z^n$. For $0\leq i < n$, set
$$\mathcal{L}_{n,i}= \vee \{z^l: l=i\; \mbox{mod}\;n\}.$$
$\{\mathcal{L}_{n,i}\}_{i=0}^{n-1}$ are minimal
reducing subspaces of $T_{z^n}$ and they are pairwise orthogonal.
Set $L_{a}^{2,i}(\mathbb{D}\setminus\{0\})=
L_{a}^2\big(\mathbb{D}\setminus\{0\},|\frac{1}{z^{\frac{n-i-1}{n}}}|^2\,dz\wedge d\bar
z\big)$, $0\leq i<n$. Define the map
$U_i:L_{a}^{2,i}(\mathbb{D}\setminus\{0\})\to \mathcal{L}_{n,i}$ as follows;
\[U_i(f)(z)=\sqrt{n}z^if(z^n), \;\;\mbox{for} \;\;z \in\mathbb{D}.\] Then $U_i$ is a unitary map such that $U_i\circ T_z =
T_{z^n}|_{\mathcal{L}_{n,i}}\circ U_i$. Thus restrictions of
$T_{z^n}$ to its minimal reducing subspaces are unitarily
equivalent to the multiplication operators $T_z$ on weighted
Bergman spaces, and there exists a distinguished minimal reducing subspace such that the restriction operator is unitarily equivalent to the Bergman shift. Further, one can calculate the curvatures of these restrictions, and prove that they are not unitarily equivalent, (c.f.\cite{ron})

 We are interested in describing model spaces for the minimal reducing subspaces
of Blashke product $B$ of order $n$. We show they are
weighted Bergman spaces attained from the analogue of bundle shifts building by using the Bergman space, and such that the multiplication operators on
them are unitarily equivalent to restrictions of $T_B$ on its minimal
reducing subspaces. The goal of this paper is trying to provide generalized Bergman bundle shifts as the model spaces viz, complex geometry and Riemann surface, etc.

To describe the progress in this direction, we need to introduce
some notations. For the \emph{flat unitary vector bundle} $E$ we
denote the multiplication operator on the space of holomorphic
$L^2$ sections of $E$, $L_a^2(E)$, by $T_E$. Set, $\mathcal{S}= B(\{\beta\in
\mathbb{D} :B^{\prime}(\beta)=0\}).$ One of our main result is the
following.

\begin{thm}
Let $B$ be a finite Blaschke product of order $n$. Let $T_B$ be
the multiplication operator on $L_a^2(\mathbb{D})$ by $B$. Let
$E_{B}$ be the flat unitary vector bundle over
$\mathbb{D}\setminus\mathcal{S}$ determined by $B$. Let $T_{E_{B}}$ be the
multiplication operator on $L_a^2(E_{B})$ defined by $z$. Then the
operator $T_B$ is unitarily equivalent to the operator $T_{E_{B}}$.
\end{thm}

The definition of $E_{B}$ will be explained in the next section.

\section{ Generalized Bundle Shifts}
For supplying the model spaces of $T_{B}$ on its reducing subspaces, the first thing of this section is to define a generalized bundle shift and then classify them under unitary equivalence. The investigation of this section extends the paper \cite{ad}.
\subsection{Generalized Bergman Bundle Shifts}
 Let $\Omega$ be an open subset of $\mathbb{C}^m$. Let $\mathfrak{K}$ be a Hilbert space.
\begin{defn} A \emph{continuous vector bundle} , which is a family of Hilbert spaces over $\Omega$, is a
topological  space $E$ together with:
\begin{enumerate}
\item A continuous map $q: E\to \Omega$,
\item A Hilbert space
structure on each fibre $E_z=q^{-1}(z)$, $z\in \Omega$, such that
the Hilbert space topology on $E_z$ is the same as the topology
inherited from $E$,
\item For each $z\in\Omega$, there exists a
neighborhood $U$ of $z$ in $\Omega$ and a homeomorphism $\phi_U: q^{-1}(U)\to
U\times \mathfrak{K}$ such that:
\begin{enumerate}
\item For each $(w,k)\in U\times \mathfrak{K}$, the point
$\phi_U(w,k)\in E_w$, \item For each $w\in U$, the map $\phi_U^w:
\mathfrak{K}\to E_w$ defined by $\phi_U^w(k)=\phi_U(w,k)$ is a
continuous linear isomorphism.
\end{enumerate}
\end{enumerate}
\end{defn}

\begin{defn}
\begin{enumerate}
\item A continuous vector bundle $E$ is called a \emph{holomorphic}
vector bundle if for all open sets $U$ and $V$ such that $U\cap
V\neq\varnothing$ then
$${\phi_U\circ \phi_{V}^{-1}}|_{(U\cap
V)\times \mathfrak{K}}: (U\cap V)\times \mathfrak{K}\to (U\cap V)\times
\mathfrak{K}$$ is given by the following expression
$$\phi_U\circ \phi_{V}^{-1}(w,k)=(w,
\phi_{UV}(w)k),$$ where $\phi_{UV}: U\cap V \to
\mathcal{G}\mathcal{L}(\mathfrak{K})$ is holomorphic and
$\mathcal{G}\mathcal{L}(\mathfrak{K})$ is the set of invertible
bounded linear operator on the Hilbert space $\mathfrak{K}$.
\item A holomorphic vector bundle $E$ is called a \emph{Hermitian} holomorphic
vector bundle if an inner product is given on $E_z$ which varies
smoothly as a function of $z$, that is, for given any two smooth
local sections $s,t$ of $E$, the function $z\mapsto \langle
s(z),t(z)\rangle_z$ is a smooth function.
\end{enumerate}
\end{defn}

 For a given  Hermitian holomorphic vector bundle $E$, let
$\Gamma_a(\Omega)$ be the space of holomorphic sections of $E$
over $\Omega$. The Bergman space $L_a^2(E)$ is defined as
\[L_a^2(E)=\{f\in \Gamma_a(\Omega):
\|f\|_{L_a^2(E)}^2:=\int_{\Omega}\|f(z)\|^2_{E_z}\,dV(z)<\infty\}.\]
Usually, a reproducing kernel Hilbert space consists of vector-valued functions taking value in $\mathbb{C}^{k}$, but more general notion is possible.
A point evaluation $ev_z: L_a^2(E)\to E_z$ is a continuous
function for all $z\in \Omega$, so $L_a^2(E)$ is a reproducing kernel Hilbert space. As a fact, the function
$K_E:\Omega\times \Omega\to \mathcal{L}(E^*,E)$ defined by $K_E(z,w)=ev_z\circ ev_w^*$, for all
$z,w\in \Omega$, takes value in
$\mathcal{L}(E_w^*,E_z)$, and for
$\eta\in E_z^*$ and $\xi\in E_w^*$ we have
\begin{eqnarray*}
\langle\eta,K_E(z,w)\xi\rangle_{E_z}&=& \langle\eta, ev_z\circ
ev_w^*\xi\rangle_{E_z}= \langle ev_z^*\eta,  ev_w^*\xi\rangle_{L_a^2(E)}\\
&=& \langle ev_w\circ ev_z^*\eta,  \xi\rangle_{E_w}=\langle K_E(w,z)\eta,\xi\rangle_{E_w}\\
&=&\langle \eta,K_E(w,z)^*\xi\rangle_{E_z},
\end{eqnarray*}
which implies  $$K_E(z,w)=K_E(w,z)^*.$$ Moreover,
\[\langle f, K_E(\cdot,w)\xi\rangle_{L_a^2(E)}=\langle f, ev_w^*\xi\rangle_{L_a^2(E)}=\langle ev_w(f),\xi\rangle_{E_w}=\langle f(w),\xi\rangle_{E_w}.\]
Hence, $K_E$ is the reproducing kernel of $L_a^2(E)$.




There is another way to view the Bergman space over the flat vector bundle. In particular, one can view it as a space of functions on the universal covering space. Sometime this representation is quite useful.

\begin{defn}
Let $E$ be a vector bundle. A \emph{unitary coordinate covering} for $E$ is
a coordinate covering $\{U,\phi_U\}$ such that for each open set
$U$ and $z\in U$, the fiber map ${\phi_U}|_{E_z}:E_z\to
\{z\}\times \mathfrak{K}$ is unitary. The unitary coordinate
covering $\{U,\phi_U\}$ is said to be \textit{flat} if the matrix
valued functions $\phi_{UV}:U\cap V\to \mathcal{U}(\mathfrak{K})$
are constants, where $\phi_U\circ
\phi_V^{-1}(z,v)=(z,\phi_{UV}(z)v)$ for $z\in U\cap V$ and
$v\in\mathfrak{K}$. A \textit{flat} unitary vector bundle is a
vector bundle equipped with a flat unitary coordinate covering.
\end{defn}

A flat structure on a vector bundle over $\Omega$ gives rise to a
representation $$\alpha:\pi_1(\Omega)\to
\mathcal{U}(\mathfrak{K})$$ of the fundamental group of $\Omega$
via parallel displacement. Conversely, suppose we have
a representation $\alpha$ and let $\tilde{\Omega}$ be the universal
covering space of $\Omega$, we can construct a flat bundle as
follows:  the equivalence relation $(z_1,v_1)\sim (z_2,v_2)$ on $\tilde{\Omega}\times
\mathfrak{K}$ is defined as if
$z_2=A(z_1)$ and $v_2=\alpha(A)v_1$ for some $A\in\pi_1(\Omega)$.
This equivalence relation gives rise to a flat vector bundle
$E_{\alpha}=\tilde{\Omega}\times \mathfrak{K}/\sim$ with natural
projection $\pi$. The vector bundle $E_{\alpha}$ constructed above
is a flat unitary vector bundle. Discussions about flat vector
bundles can be found in \cite{ad, chen,kob}. The group
$\mathcal{U}(\mathfrak{K})$ acts on
$\mbox{Hom}(\pi_1(\Omega),\mathcal{U}(\mathfrak{K}))$ via
conjugation $(V,\alpha)\to V\alpha V^*$. The set of equivalence
classes is denoted by
$\mbox{Hom}(\pi_1(\Omega),\mathcal{U}(\mathfrak{K}))/\mathcal{U}(\mathfrak{K})$.
\begin{pro}(\cite{ad})There is one to one correspondence
between
\mbox{$\mbox{Hom}(\pi_1(\Omega),\mathcal{U}(\mathfrak{K}))/\mathcal{U}(\mathfrak{K})$}
and the set of equivalence classes of flat unitary vector bundles
over $\Omega$.
\end{pro}

For a flat unitary vector bundle $E$, the $m$-tuple
multiplication operator $(T_{z_1},\ldots,T_{z_m})$ on $L_a^2(E)$ is denoted
by $T_E$.

\begin{defn}
Let $X$ and $Y$ be two topological spaces. A continuous map $q:Y\to
X$ is a \emph{covering map} if there is a neighborhood
$U$ for each point $x\in X$ such that;
\begin{enumerate}
\item the set $q^{-1}(U)$ is a disjoint union of open sets $V_i$,
\item the restriction map of $q$ to $V_i$ is a homeomorphism of
$V_i$ onto $U$.
\end{enumerate}
\end{defn}

The group of covering transformations for $q$ is the group of
homeomorphisms $A$ of $Y$ such that $q\circ A=q$. Two covering
spaces $(Y_1,q_1)$ and $(Y_2,q_2)$ of $X$ are said to be
\emph{equivalent} if there is a homeomorphism $\phi$ from $Y_1$ onto
$Y_2$ such that $q_2\circ \phi=q_1$.

Let $q: \tilde{\Omega}\to \Omega$ be a universal covering map.
Let $J_{q}(z)=((\tfrac{\partial q_i}{\partial z_j}(z)))_{i,j=1}^{m}$ be Jacobin of $q=(q_1,\ldots,q_m)$. For $\alpha\in
\mbox{Hom}(\pi_1(\Omega),\mathcal{U}(\mathfrak{K}))$,
$L_a^2(\tilde{\Omega},\mathfrak{K},\alpha,|\det\,J_q|^2\,dV(z))$ is defined
to be space of functions $f$ in
$L_a^2(\tilde{\Omega},\mathfrak{K},|\det\,J_q(z)|^2\,dV(z))$ such that
\[\int_{\tilde{\Omega}}\|f(z)\|_{\mathfrak{K}}^2|\det\,J_q|^2\,dV(z)<\infty\]
and $f(A(z))=\alpha(A)f(z)$ for all $A\in\pi_1(\Omega)$ and
$z\in\tilde{\Omega}$. A bounded holomorphic function $\phi$
defined on $\tilde{\Omega}$ is said to be
\emph{$\pi_1(\Omega)$-automorphic} if $\phi\circ A=\phi$ for each $A$ in
$\pi_1(\Omega)$. Let $H^{\infty}(\tilde{\Omega},\pi_1(\Omega))$
denote the set of all $\pi_1(\Omega)$-automorphic bounded
holomorphic functions. The space
$L_a^2(\tilde{\Omega},\mathfrak{K},\alpha,|\det\,J_q|^2\,dV(z))$
is invariant for $H^{\infty}(\tilde{\Omega},\pi_1(\Omega))$, that
is, if $\phi\in H^{\infty}(\tilde{\Omega},\pi_1(\Omega))$ and
$f\in
L_a^2(\tilde{\Omega},\mathfrak{K},\alpha,|\det\,J_q|^2\,dV(z))$
then $\phi f\in
L_a^2(\tilde{\Omega},\mathfrak{K},\alpha,|\det\,J_q|^2\,dV(z))$.
In particular, it is invariant for multiplication by the
components $q_i, 1\leq i\leq m$, of covering map $q$ and we define
the operator $T_{\alpha}=(T_{\alpha_1},\ldots,T_{\alpha_m})$ on
$L_a^2(\tilde{\Omega},\mathfrak{K},\alpha,|\det\,J_q|^2\,dV(z))$
by $T_{\alpha_i}(f)=q_i f$, for $1\leq i\leq m$.

\begin{pro}\label{b1}
If $\alpha$ is in
$\mbox{Hom}(\pi_1(\Omega),\mathcal{U}(\mathfrak{K}))$ and
$E_{\alpha}$ is the flat unitary bundle over $\Omega$ determined
by $\alpha$, then $T_{\alpha}$ is unitarily equivalent to the
bundle shift $T_{E_{\alpha}}$.
\end{pro}
\begin{proof}
Define a map
$F:L_a^2(\tilde{\Omega},\mathfrak{K},\alpha,|\det\,J_q|^2\,dV(z))\to
L_a^2(E_{\alpha})$ as  follows $$F(f)(w)= [(z,f(z))]\;\;\;
\mbox{for} \;\;z\in q^{-1}(w),\; w\in\Omega,$$ where $[(z,f(z))]$
denote equivalence class of $(z,f(z))$. It is easy to see that
$$\|F(f)(w)\|_{(E_{\alpha})_w}^2=\|f(z)\|^2_{\mathfrak{K}}.$$

\noindent (1) $F$ is an isometry:
\begin{eqnarray*}
\|F(f)\|^2_{L_a^2(E)}&=&
\int_{\Omega}\|F(f)(w)\|_{(E_{\alpha})_w}^2\;\;dV(w)\\
&=&\int_{\tilde{\Omega}}\|F(f)(q(z))\|_{(E_{\alpha})_w}^2|\det\;J_q(z)|^2\;\;dV(z)\\
&=&\int_{\tilde{\Omega}}\|f(z)||_{\mathfrak{K}}^2|\det\;J_q(z)|^2\;\;dV(z)\\
&=&\|f\|^2_{L_a^2(\tilde{\Omega},\mathfrak{K},\alpha,|\det\,J_q|^2\,dV(z))}
\end{eqnarray*}
(2) $F$ is onto: Let $s$ be a section of $E_{\alpha}$ such that
$s\in L_a^2(E_{\alpha})$. Define function
$f_s:\mathbb{D}\to\mathfrak{K}$ such that $s(q(z))=[(z,f_s(z))]$
for $z\in\tilde{\Omega}$. One can show that $f_s$ is holomorphic
and $q(z)=q(A(z))$ implies that $f_s(A(z))=\alpha(A)f_s(z)$ for
all $A\in\pi_1(\Omega)$ and $z\in\tilde{\Omega}$. One easily see
that
$$\|s\|^2_{L_a^2(E_{\alpha})}=\int_{\tilde{\Omega}}\|f_s(z)\|_{\mathfrak{K}}^2|\det\,J_q(z)|^2\;\;dV(z).$$ Hence $f_s\in
L_a^2(\tilde{\Omega},\mathfrak{K},\alpha,|\det\,J_q|^2\,dV(z))$
and $F(f_s)=s$.

\noindent(3) $F$ intertwines $T_{\alpha}$ and $T_{E_{\alpha}}$: for
$f\in
L_a^2(\tilde{\Omega},\mathfrak{K},\alpha,|\det\,J_q|^2\,dV(z))$,
$w\in\Omega$, $z\in q^{-1}(w)$ and $1\leq i\leq m$
\begin{eqnarray*}
(F\circ T_{\alpha_i})(f)(w)&=& F(q_i f)(w)=[(z,(q_i f)z)]=[(z,q_i(z)f(z))]\\
&=&[(z,w_i f(z))]=w_i[z,f(z)]=w_i F(f)(w)\\
&=& (T_{z_i}\circ F) (f)(w).
\end{eqnarray*}
Thus $F\circ T_{\alpha_i}=T_{z_i}\circ F$.
\end{proof}

\subsection{Bergman Bundle Shifts}
Although most of results in this part can be extended to multivariable cases or weighted Bergman spaces, we restrict ourself to one variable case for further use. Moreover, when the domain has nice boundary, we can follow the approach of Abrahamse and the the first author which we do in this section. In section 2.4, we show how to provide an alterative approach covering the general case.

 Let $R$ be a multiply connected domain in $\mathbb{C}$ whose boundary consists of $n+1$ analytic Jordan curves and $\tilde{q}:\mathbb{D}\to R$ be the covering
map. Let $G$ be the group of covering transformations of $\tilde{q}$.
The group $G$ is isomorphic to fundamental group of $R$.

A function $f$ on $\mathbb{D}$ is said to be \emph{$G$-automorphic} if
$f\circ A=f$ for each $A$ in $G$. Let
$L^2(\mathbb{D},\mathbb{C}^n, G)$ be a subspace of $G$-automorphic
functions in $L^2(\mathbb{D},\mathbb{C}^n)$ and let
$L_a^2(\mathbb{D},\mathbb{C}^n, G)$ be a subspace of $G$-automorphic functions in $L_a^2(\mathbb{D},\mathbb{C}^n)$. Let
$L^{\infty}(\mathbb{D},G)$ be the algebra of $G$-automorphic
functions in $L^{\infty}(\mathbb{D})$ and let
$H^{\infty}(\mathbb{D},G)$ be the algebra of $G$-automorphic
functions in $H^{\infty}(\mathbb{D})$.

\begin{pro}\label{inp1}
\begin{enumerate}
\item The smallest invariant subspace for
$L^{\infty}(\mathbb{D},G)$ containing
$L_a^2(\mathbb{D},\mathbb{C}^n,G)$ is
$L^2(\mathbb{D},\mathbb{C}^n,G)$, where $L^{\infty}(\mathbb{D},G)$ acts on $L^2(\mathbb{D},\mathbb{C}^n,G)$ pointwisely.

\item The smallest invariant subspace for $H^{\infty}(\mathbb{D})$
containing $L_a^2(\mathbb{D},\mathbb{C}^n,G)$ is
$L_a^2(\mathbb{D},\mathbb{C}^n)$, where $H^{\infty}(\mathbb{D})$ acts on $L_a^2(\mathbb{D},\mathbb{C}^n)$ pointwisely.

\item The smallest invariant subspace for $L^{\infty}(\mathbb{D})$
containing $L_a^2(\mathbb{D},\mathbb{C}^n,G)$ is
$L^2(\mathbb{D},\mathbb{C}^n)$, where $L^{\infty}(\mathbb{D})$ acts on $L^2(\mathbb{D},\mathbb{C}^n)$ pointwisely.
\end{enumerate}
\end{pro}
\begin{proof}
Since $L_a^2(\mathbb{D},\mathbb{C}^n,G)$ contains the constant
functions, the proof is immediate.
\end{proof}

 An operator $T$ on
$L_{\mathbb{C}^n}^{2}(\mathbb{D})$ will be called to be \emph{decomposable} if
there is a weakly measurable function $w\mapsto T_{w}$ from
$\mathbb{D}$ to $B(\mathbb{C}^n)$ such that for $f\in
L_{\mathbb{C}^n}^{2}(\mathbb{D})$, $(Tf)(w)=T_{w}f(w)$, a.e. $dv$.
It is well known that an operator on
$L_{\mathbb{C}^n}^{2}(\mathbb{D})$ is decomposable if and only if
it commutes with $L^{\infty}(\mathbb{D})$. Moreover,
 the algebraic operations can be performed pointwisely, that is
 $[(T+S)f](w)=T_{w}f(w)+S_{w}f(w)$, $(T^{\ast}f)(w)=T_{w}^{\ast}f(w)$,
 etc., and the norm is $\|T\|=\mbox{ess}\,\sup\|T_{w}\|$.
 One can see \cite{d, d1} for further references.

\begin{pro}\label{dhp}
For every unitary representation $\alpha$ of $G$ on
$\mathbb{C}^n$, there is a decomposable operator $\Phi_{\alpha}$
on $L_{\mathbb{C}^n}^{2}(\mathbb{D})$ such that
$\Phi_{\alpha}(L^{2}_{a}(\mathbb{D},\mathbb{C}^n,G))=L_{a}^{2}(\mathbb{D},\mathbb{C}^n,\alpha)$.
Moreover, $\Phi_{\alpha}$ satisfies
$\Phi_{\alpha ,g(\lambda)}=\alpha(g)\Phi_{\alpha,\lambda}$ a.e.
$dm$ for every $g\in G$.
\end{pro}
\begin{proof}
Note that the covering map $q:\mathbb{D}\rightarrow R$ can be
extended to a covering map from an open set
$\widetilde{\mathbb{D}}$ containing $\mathbb{D}$ to an open set
$\widetilde{R}$ containing the closure of $R$. Given $\alpha$, let $E_{\alpha}$
be the flat unitary bundle over $\widetilde{\mathbb{D}}$, and let
$\pi$ denote the projection from
$\widetilde{\mathbb{D}}\times\mathbb{C}^n$ onto $E_{\alpha}$.
$E_{\alpha}$ is analytically equivalent to the trivial bundle
$\widetilde{R}\times\mathbb{C}^n$, so there exists an analytic
isomorphism $\phi$ from $E_{\alpha}$ to
$\widetilde{R}\times\mathbb{C}^n$. It is clear that the
composition $\phi\circ \pi$ is a bundle map from
$\widetilde{\mathbb{D}}\times\mathbb{C}^n$ onto
$\widetilde{R}\times\mathbb{C}^n$ which yields an analytic
function $\Phi$ from $\widetilde{\mathbb{D}}$ to
$\mathcal{GL}(\mathbb{C}^n)$, the set of all invertible operators
on $\mathbb{C}^n$, such that $\Phi(g(z))=\alpha(g)\Phi(z)$ for all
$z\in\widetilde{\mathbb{D}}$ and $g\in G$. Let $\Phi_{\alpha}$ be
the restriction of $\Phi$ to the closed disk, it follows that
there is a $K>0$ such that $|\Phi_{\alpha}(z)|\leq K$ and
$|\Phi_{\alpha}(z)^{-1}|\leq K$ for all $z$ in the closed disk.
Let $f\in L^2_{\mathbb{C}^n}(\mathbb{D})$, consider
\begin{eqnarray*}
\int_{\mathbb{D}}|(\Phi_{\alpha} f)(z)|^2 dm&=&-\frac{1}{2\pi
i}\int_{\mathbb{D}}|(\Phi_{\alpha} f)(z)|^2dz\wedge d\bar{z}\\
&=& -\frac{1}{2\pi
i}\int_{\mathbb{D}}|\Phi_{\alpha}(z) f(z)|^2dz\wedge d\bar{z}\\
&\leq& -\frac{K}{2\pi
i}\int_{\mathbb{D}}|f(z)|^2dz\wedge d\bar{z}\\
&=& K\|f\|^2_{L^2_{\mathbb{C}^n(\mathbb{D})}}
\end{eqnarray*}

Hence, $\Phi_{\alpha}$ is an invertible decomposable operator on
$L^2_{\mathbb{C}^n(\mathbb{D})}$. Since $\Phi_{\alpha}$ is
analytic and $\Phi_{\alpha}(z)$ and $\Phi_{\alpha}(z)^{-1}$ are uniformly
bounded for $z\in \bar{\mathbb{D}}$, so we have
$\Phi_{\alpha}(L^{2}_{a}(\mathbb{D},\mathbb{C}^n,G))=L_{a}^{2}(\mathbb{D},\mathbb{C}^n,\alpha)$.
\end{proof}

\begin{pro}
\begin{enumerate}
\item The smallest invariant subspace for
$L^{\infty}(\mathbb{D},G)$ containing
$L_a^2(\mathbb{D},\mathbb{C}^n,\alpha)$ is
$L^2(\mathbb{D},\mathbb{C}^n,\alpha)$.

\item The smallest invariant subspace for $H^{\infty}(\mathbb{D})$
containing $L_a^2(\mathbb{D},\mathbb{C}^n,\alpha)$ is
$L_a^2(\mathbb{D},\mathbb{C}^n)$.

\item The smallest invariant subspace for $L^{\infty}(\mathbb{D})$
containing $L_a^2(\mathbb{D},\mathbb{C}^n,\alpha)$ is
$L^2(\mathbb{D},\mathbb{C}^n)$.
\end{enumerate}
\end{pro}
\begin{proof}
The proof follows from Proposition \ref{inp1} and Proposition
\ref{dhp}.
\end{proof}

\begin{lem}
An operator on $L^{2}(\mathbb{D},\mathbb{C}^n,G)$ commutes with
$L^{\infty}(\mathbb{D},G)$ if and only if it is a $G$-automorphic
decomposable operator.
\end{lem}
\begin{proof}
The comnutant of $L^{\infty}(\Omega)$ defined on
$L^{2}(\Omega,\mathbb{C}^n)$ consists of decomposable operators on
$L^{2}(\Omega,\mathbb{C}^n)$ \cite{d,d1}. Since
$L^{2}(\mathbb{D},\mathbb{C}^n,G)$ is isometrically isomorphic to
$L^{2}(\Omega,\mathbb{C}^n)$, the lemma follows.
\end{proof}

\begin{pro}\label{dt}
An operator $W:L^{2}(\mathbb{D},\mathbb{C}^n,\alpha)\rightarrow
L^{2}(\mathbb{D},\mathbb{C}^n,\beta)$ intertwines
$L^{\infty}(\mathbb{D},G)$ if and only if $W$ is decomposable and
$W_{g(\lambda)}=\beta(g)W_{\lambda}\alpha(g)^{\ast}$ a.e. $dv$ for
all $g\in G$.
\end{pro}
\begin{proof}
Given $W$, the operator $(\Phi_{\beta})^{-1}W\Phi_{\alpha}$ on
$L^{2}(\mathbb{D},\mathbb{C}^n,G)$ commutes with
$L^{\infty}(\mathbb{D},G)$. Thus, by the lemma, there is a
$G$-automorphic decomposable operator $V$ such that
$V=(\Phi_{\beta})^{-1}W\Phi_{\alpha}$. Hence
$W=\Phi_{\beta}V(\Phi_{\alpha})^{-1}$ and thus $W$ is
decomposable. Moreover,
\begin{eqnarray*}W_{g(\lambda)}&=&(\Phi_{\beta,g(\lambda)})V_{g(\lambda)}(\Phi_{\alpha,g(\lambda)})^{-1}
\\&=&\beta(g)\Phi_{\beta,\lambda}V_{\lambda}(\alpha(g)\Phi_{\alpha,\lambda})^{-1}\\
&=&\beta(g)\big(\Phi_{\beta,\lambda}V_{\lambda}(\Phi_{\alpha,\lambda})^{-1}\big)(\alpha(g))^{-1}\\
&=&\beta(g)W_{\lambda}\alpha(g)^{\ast}.\end{eqnarray*} The
sufficient condition is obivious.
\end{proof}

We state a series of lemmas, culminating in characterization of reducing subspaces of Bergman bundle shifts.

\begin{lem}\label{lh1}
If a decomposable operator $T$ on $L^{2}(\mathbb{D},\mathbb{C}^n)$
such that $T(L_{a}^{2}(\mathbb{D},\mathbb{C}^n))\subseteq
L_{a}^{2}(\mathbb{D},\mathbb{C}^n)$ then $F: \mathbb{D}\to
\mathbb{C}^n$, defined as $F(\lambda)= T_{\lambda}u$, is a
holomorphic map, for each fixed $u\in \mathbb{C}^n$.
\end{lem}
\begin{proof}
For $u\in \mathbb{C}^n$, consider the constant function $f(\lambda)=u$
for all $\lambda\in \mathbb{D}$. Clearly, $f\in
L_{a}^{2}(\mathbb{D},\mathbb{C}^n)$. By the given condition,
$T(f)\in L_{a}^{2}(\mathbb{D},\mathbb{C}^n)$. In particular,
$T(f)$ is a holomorphic function on $\mathbb{D}$. For any
$\lambda_0\in \mathbb{D}$,
$\lim\limits_{\lambda\to\lambda_0}\tfrac{T(f)(\lambda)-T(f)(\lambda_0)}{\lambda-\lambda_0}$
exists. Now,
\begin{eqnarray*}
\tfrac{T(f)(\lambda)-T(f)(\lambda_0)}{\lambda-\lambda_0}&=&
\tfrac{T_{\lambda}(f(\lambda))-T_{\lambda_0}(f(\lambda_0))}{\lambda-\lambda_0}\\
&=& \tfrac{T_{\lambda}(u)-T_{\lambda_0}(u)}{\lambda-\lambda_0}\\
&=& \tfrac{F(\lambda)-F(\lambda_0)}{\lambda-\lambda_0}.
\end{eqnarray*}
Hence
$\lim\limits_{\lambda\to\lambda_0}\tfrac{F(\lambda)-F(\lambda_0)}{\lambda-\lambda_0}$
exists. Thus $F$ is a holomorphic function on $\mathbb{D}$.
\end{proof}

\begin{lem}\label{lh2}
If a decomposable operator $T$ on $L^{2}(\mathbb{D},\mathbb{C}^n)$
such that $T^*(L_{a}^{2}(\mathbb{D},\mathbb{C}^n))\subseteq
L_{a}^{2}(\mathbb{D},\mathbb{C}^n)$ then $\tilde{F}: \mathbb{D}\to
\mathbb{C}^n$, defined as $\tilde{F}(\lambda)= T_{\lambda}^*u$, is
a holomorphic map, for each fixed $u\in \mathbb{C}^n$.
\end{lem}
\begin{proof}
The proof is same as proof of Lemma \ref{lh1}.
\end{proof}

\begin{lem}\label{ct}
A decomposable operator $T$ on $L^{2}(\mathbb{D},\mathbb{C}^n)$
commutes with the projection onto
$L_{a}^{2}(\mathbb{D},\mathbb{C}^n)$ if and only if it is a
constant.
\end{lem}
\begin{proof}
It is sufficient clearly. On the other side, assume that
$L_{a}^{2}(\mathbb{D},\mathbb{C}^n)$ reduces $T$, that is,
$T(L_{a}^{2}(\mathbb{D},\mathbb{C}^n))\subseteq
L_{a}^{2}(\mathbb{D},\mathbb{C}^n)$ and
$T^*(L_{a}^{2}(\mathbb{D},\mathbb{C}^n))\subseteq
(L_{a}^{2}(\mathbb{D},\mathbb{C}^n))$.

 Fix $u$ and $v$ in
$\mathbb{C}^{n}$ and define $h(\lambda)=\langle
T_{\lambda}u,v\rangle$. Since $L_{a}^{2}(\mathbb{D},\mathbb{C}^{n})$
is invariant for $T$, so by Lemma \ref{lh1}, the function $h$ is
holomorphic on $\mathbb{D}$. Similarly, since
$L_{a}^{2}(\mathbb{D},\mathbb{C}^{n})$ is invariant for $T^{\ast}$,
so by Lemma \ref{lh2}, the function
$\tilde{h}(\lambda)=\overline{h(\lambda)}$ is holomorphic on
$\mathbb{D}$. Hence, $h$ is a constant function on $\mathbb{D}$.
Since $u$ and $v$ are arbitrary, so  $\lambda \mapsto T_{\lambda}$
is a constant function from $\mathbb{D}$ to
$\mathcal{L}(\mathbb{C}^n)$ also.
\end{proof}

\begin{thm}\label{b2}
The operators $T_{\alpha}$ and $T_{\beta}$ are unitarily
equivalent if and only if $\alpha$ and $\beta$ are unitarily
equivalent.
\end{thm}
\begin{proof}
It is easy to check the sufficiency. Conversely, suppose that $U$
is an isometry from $L_{a}^{2}(\mathbb{D},\mathbb{C}^n,\alpha)$
onto $L_{a}^{2}(\mathbb{D},\mathbb{C}^n,\beta)$ with
$UT_{\alpha}=T_{\beta}U$. Then, by the fact that the minimal normal extension for a subnormal operator is unique, one can
extend $U$ to a unitary equivalence of their minimal normal
extensions. The minimal extension of $T_{\alpha}$ is
multiplication by $q$ on $L^{2}(\mathbb{D},\mathbb{C}^n,\alpha)$,
and the extension of $U$ is a unitary operator $\widetilde{U}$
from $L^{2}(\mathbb{D},\mathbb{C}^n,\alpha)$ to
$L^{2}(\mathbb{D},\mathbb{C}^n,\beta)$, and it intertwines
$L^{\infty}(\mathbb{D},G)$. Thus $\widetilde{U}$ is
decomposable and
$\widetilde{U}(L_{a}^{2}(\mathbb{D},\mathbb{C}^n)=L_{a}^{2}(\mathbb{D},\mathbb{C}^n)$.
Hence $L_{a}^{2}(\mathbb{D},\mathbb{C}^n)$ reduces
$\widetilde{U}$, it is a constant. Thus there is a unitary
operator $V$ on $\mathbb{C}^n$ such that
$\widetilde{U}_{\lambda}=V$ a.e. $dv$. But
$V=\widetilde{U}_{g(\lambda)}=\beta(g)\widetilde{U}_{\lambda}\alpha(g)^{\ast}=\beta(g)V\alpha(g)^{\ast}$
a.e. $dv$, and thus $\alpha(g)=V^{\ast}\beta(g)V$, which shows
that $\alpha$ and $\beta$ are unitarily equivalent.
\end{proof}

\begin{cor} If $E$ and $E^{\prime}$ are flat unitary bundles over $\Omega$, then operators
$T_{E}$ and $T_{E^{\prime}}$ are unitarily equivalent if and only
if $E$ and $E^{\prime}$ are equivalent as flat unitary bundles.
\end{cor}
\begin{proof}
First Suppose that the flat unitary vector bundles $E$ and $E'$ are
equivalent. Let $\Theta : E\to E'$ be the bundle isomorphism.
Then, the map $f(z)\mapsto \Theta(z) f(z)$ for $z\in\Omega$ defines a module isomorphism
between $A(\Omega)$ modules $\Gamma_a(E)$ and $\Gamma_a(E')$. For
$z\in \Omega$, $\|\Theta(z)f(z)\|^2_{E'_z}=\|f(z)\|^2_{E_z}$. If
$f\in L_a^2(E)$ then
\begin{eqnarray*}
\|\Theta f\|^2_{L_a^2(E')}&=&-\frac{1}{2\pi
i}\int_{\Omega}\|\Theta(z)f(z)\|^2_{E'_z}\,dz\wedge d\bar z\\
&=&-\frac{1}{2\pi
i}\int_{\Omega}\|f(z)\|^2_{E_z}\,dz\wedge d\bar z\\
&=& \| f\|^2_{L_a^2(E)}.
\end{eqnarray*}
Hence $\Theta$ defines an isometry from $L_a^2(E)$ onto $L_a^2(E')$.
Since $\Theta$ is a module isomorphism, this intertwines $T_E$ and
$T_{E'}$.

Conversely suppose that operators $T_E$ and $T_{E'}$ are unitarily
equivalent, so by Proposition \ref{b1} and Theorem \ref{b2},
bundle $E$ and $E'$ are unitarily equivalent as flat unitary
vector bundles.
\end{proof}

If $\Omega_{1}$ and $\Omega_{2}$ are two multiply connected domains in the complex plane whose boundaries consist of Jordan curves, and biholomorphic to each other by the map $\varphi$. $E$ is a flat unitary bundle over $\Omega_{2}$, then the pull back bundle $\varphi_{1}^{\ast}E$ is a flat unitary bundle over $\Omega_{1}$.

\begin{cor} The bundle shifts
$T_{E}$ and $T_{\varphi^{\ast}E}$ are unitarily equivalent.
\end{cor}

\subsection{The commutant of $W^*(T_{\alpha})$}
We recall some results about the relation between subnormal operator and its normal extension (cf. \cite{con}). Let $S$ be a subnormal operator on a Hilbert space $\mathcal{H}$
and let $N$ be its minimal normal extension on the Hilbert space
$\mathcal{K}$. If an operator commute with operator $S$ and $S^*$,
that is, the operator in the commutant of the $W^*$ algebra
$\mathcal{W}^*(S)$ generated by $S$, lifts to operators that
commute with $N$ (cf. \cite{bram}).

Let $P:\mathcal{K}\to \mathcal{H}$ be the orthogonal projection. Let
$\mathcal{A}$ denote the algebra of operators on $\mathcal{K}$
that commute with $N$ and $P$, and let $(\mathcal{W}^*(S))'$
denote commutants of $\mathcal{W}^*(S)$.
\begin{thm}\label{comm}
If $E$ is a $n$-dimensional flat unitary bundle over $\Omega$, then the algebra of commutants of $T_{E}$ is $H^{\infty}_{L(E)}(\Omega)$.
\end{thm}
\begin{proof}
Every analytic vector bundle over $\Omega$ is analytically trivial by Grauert theorem, and we assume that $V:\Omega\times\mathbb{C}^{n}\rightarrow E$ is an analytic isomorphism. Then $V$ defines the similarity $\widetilde{V}:L_{a}^{2}(E)\rightarrow L_{a}^{2}(\Omega,\mathbb{C}^{n})$ between $T_{\mathbb{C}^{n}}$ and $T_{E}$, that is, $T_{E}=\widetilde{V}^{-1}T_{\mathbb{C}^{n}}\widetilde{V}$. If $L_{a}^{2}(\Omega,\mathbb{C}^{n})$ is rewriten as $L_{a}^{2}(\Omega)\otimes\mathbb{C}^{n}$, then $T_{\mathbb{C}^{n}}=T_{z}\otimes I$, where $I$ is the identity operator on $\mathbb{C}^{n}$. $A$ is an operator on $L_{a}^{2}(E)$ commuting with $T_{E}$, then
\[A\widetilde{V}^{-1}(T_{z}\otimes I)\widetilde{V}=AT_{E}=T_{E}A=\widetilde{V}^{-1}(T_{z}\otimes I)\widetilde{V}A.\]
This means $(T_{z}\otimes I)(\widetilde{V}^{-1}A\widetilde{V})=(\widetilde{V}^{-1}A\widetilde{V})(T_{z}\otimes I)$, so $(\widetilde{V}^{-1}A\widetilde{V})\in (T_{z}\otimes I)'=H^{\infty}(\Omega)\otimes M_{n}(\mathbb{C}^{n})$. It follows that $A\in \widetilde{V}(H^{\infty}(\Omega)\otimes M_{n}(\mathbb{C}^{n}))\widetilde{V}^{-1}=H^{\infty}_{L(E)}(\Omega)$.\\
\indent
Conversely, it is clear that every $\Phi\in H^{\infty}_{L(E)}(\Omega)$ commutes with $T_{E}$, and the proof is completed.
\end{proof}
\begin{thm}\cite{bram}
The map $A\mapsto A|_{\mathcal{H}}$ is a $*$-isometric
isomorphism from $\mathcal{A}$ onto $(\mathcal{W}^*(S))'$.
\end{thm}

Let $\mathcal{W}^*(\alpha)$ be the $W^*$-subalgebra of
$\mathcal{L}(\mathbb{C}^n)$ generated by $\alpha(G)$.
\begin{thm}
There is a $*$-isometric isomorphism from
$(\mathcal{W}^*(\alpha))'$ onto $(\mathcal{W}^*(T_{\alpha}))'$.
\end{thm}
\begin{proof}
Let $\mathcal{U}$ be the algebra of operators on
$L^2(\mathbb{D},\mathbb{C}^n,\alpha)$ that commute with
$L^{\infty}(\mathbb{D},G)$ and with the projection onto
$L_a^2(\mathbb{D},\mathbb{C}^n,\alpha)$. For $B$ in
$(\mathcal{W}^*(\alpha))'$ and $f$ in $L^2(\mathbb{C}^n,\alpha)$,
observe that $(Bf)\circ A= B(f\circ A)=B(\alpha(A)f)=\alpha(A)Bf$,
for all $A$ in $G$. Thus, there is a decomposable
$\tau_{\alpha}(B)$ on $L^2(\mathbb{D},\mathbb{C}^n)$ defined by
$\tau_{\alpha}(B)(f)=Bf$.

Since a constant decomposable operator maps $L_a^2(\mathbb{D})$
into $L_a^2(\mathbb{D})$ and since $\tau_{\alpha}(B)$ maps
$L^2(\mathbb{D},\mathbb{C}^n,\alpha)$ into
$L^2(\mathbb{D},\mathbb{C}^n,\alpha)$, thus it also maps
$L_a^2(\mathbb{D},\mathbb{C}^n,\alpha)$ into
$L_a^2(\mathbb{D},\mathbb{C}^n,\alpha)$. In other words, it takes
$(\mathcal{W}^*(\alpha))'$ into $\mathcal{U}$. Clearly,
$\tau_{\alpha}(B^*)f=B^*f=(\tau_{\alpha}B)^*f$ and
$\|\tau_{\alpha}B\|=\mbox{ess}\,\sup_{\lambda}\|{(\tau_{\alpha}(B))}_\lambda\|=\|B\|$.
Thus $\tau_{\alpha}$ is $*$- isometric map. If $X$ is in
$\mathcal{U}$, then by Proposition \ref{dt}, the operator $X$ is
decomposable $X_{g(\lambda)}=\alpha(g)X_{\lambda}\alpha(g)^*$ for
all $g\in G$. Since $X$ is reduced by
$L^2_a(\mathbb{D},\mathbb{C}^n,\alpha)$, so
$X(L^2_a(\mathbb{D},\mathbb{C}^n,\alpha))\subseteq
L^2_a(\mathbb{D},\mathbb{C}^n,\alpha)$ and
$X^*(L^2_a(\mathbb{D},\mathbb{C}^n,\alpha))\subseteq
L^2_a(\mathbb{D},\mathbb{C}^n,\alpha)$. Thus, both $X$ and $X^*$
maps $L_a^2(\mathbb{D},\mathbb{C}^n)$ to
$L_a^2(\mathbb{D},\mathbb{C}^n)$. Hence by Lemma \ref{ct}, $X$
is a decomposable operator, and so it is a constant, thus there is a
$B$ in $\mathcal{L}(\mathbb{C}^n)$ with $X_{\lambda}=B$ a.e. $dm$.
Now,
\[B=X_{g(\lambda)}=\alpha(g)X_{\lambda}\alpha(g)^*=\alpha(g)B\alpha(g)^*.\]
Thus operator $B$ is in $(\mathcal{W}^*(\alpha))'$ and so
$X=\tau_{\alpha}(B)$. This proves that $\tau_{\alpha}$ is onto.
\end{proof}

Note that the algebra $W^{\ast}(T_{\alpha})'$ is finite dimensional for a Bergman bundle shift even though the Bergman space $L_{a}^{2}(E)$ is infinite dimensional.

\subsection{The case of bad boundary}
In section 3, we need similar results about bundle shifts over multiply connected domains to $D_{0}=\mathbb{D}\setminus\mathcal{S}$ obtained by removing a finite subset from $\mathbb{D}$. Although, the proceeding results don't apply directly. In this case, the Bergman bundle shift is a Cowen-Douglas operator on $D_{0}$, we can prove similar results by applying ideas from \cite{cd}. In particular, $D_{0}$ is a multiply connected domain with fundamental group $\pi_{1}(D_{0})$, and its universal covering is $\mathbb{D}$. $\mathbb{D}\times\mathbb{C}^{n}$ is the trivial bundle over $\mathbb{D}$, $E_{\alpha}$ is the corresponding flat unitary vector bundle over $D_{0}$ defined by a representation $\alpha:\pi_{1}(D_{0})\rightarrow\mathcal{U}_{n}$. The Bergman space $L_{a}^{2}(E,\alpha)$ of holomorphic sections on $E_{\alpha}$ is defined similarly. It can be characterized as the space of all $\alpha$-automorphic functions in $L^{2}_{a}(\mathbb{D},\mathbb{C}^{n})$, denoted by $L^{2}_{a}(\mathbb{D},\mathbb{C}^{n},\alpha)$. And we define the Bergman bundle shift $T_{\alpha}$ also. For $D_{0}$, we have similar arguments and establish results which are similar to Theorem \ref{b2}.
\begin{thm}\label{thmb}
There is a onto $*$-isometric isomorphism between
$(\mathcal{W}^*(T_{\alpha}))'$ and $(\mathcal{W}^*(\alpha))'$.
\end{thm}
\begin{proof}
If $A$ is an operator commuting with $T_{\alpha}$, then there exists a bundle map $\Phi\in H^{\infty}_{M_{n}(\mathbb{C})}(\mathbb{D})$ such that $Af(z)=\Phi(z)f(z)$ for every $f\in L^{2}_{a}(\mathbb{D},\mathbb{C}^{n},\alpha)$. (Note the techniques for the Cowen-Douglas operator carry over to flat bundle case, cf \cite{cd}). And for $f\in L^{2}_{a}(\mathbb{D},\mathbb{C}^{n},\alpha)$, $\gamma\in\pi_{1}(D_{0})$, we have
\[(\Phi f)(\gamma z)=\Phi(\gamma z)f(\gamma z)=\Phi(\gamma z)\alpha(\gamma)f(z),\]
and
\[(\Phi f)(\gamma z)=\alpha(\gamma)(\Phi f)(z)=\alpha(\gamma)\Phi(z)f(z)\]
since $(\Phi f)\in L^{2}_{a}(\mathbb{D},\mathbb{C}^{n},\alpha)$. Then
\[\Phi(\gamma z)=\alpha(\gamma)\Phi(z)\alpha(\gamma)^{\ast}.\]
\indent
For every projection $P\in (\mathcal{W}^*(T_{\alpha}))'$, there a $\Psi\in H^{\infty}_{M_{n}(\mathbb{C})}(\mathbb{D})$ such that $Pf(z)=\Psi(z)f(z)$ for every $f\in L^{2}_{a}(\mathbb{D},\mathbb{C}^{n},\alpha)$. $P^{2}=P$ and $P^{\ast}=P$ implies that $\Psi(z)^{\ast}=\Psi(z)$ and $\Psi(z)^{2}=\Psi(z)$ for $z\in\mathbb{D}$, and so $\Psi(z)$ equals a constant projection $Q$ since $\Psi(z)$ and $\Psi(z)^{\ast}$ is analytic. Moreover,
\[\Psi(\gamma z)=\alpha(\gamma)\Psi(z)\alpha(\gamma)^{\ast},\text{ for }\gamma\in\pi_{1}(D_{0}),\]
so,
\[Q=\alpha(\gamma)Q\alpha(\gamma)^{\ast}.\]
that is, $Q\in (\mathcal{W}^*(\alpha))'$. Conversely, it is clear that $\Psi(z)=Q\in(\mathcal{W}^*(T_{\alpha}))'$ for every $Q\in (\mathcal{W}^*(\alpha))'$, and the correspondence is one-to-one and $*$.
\end{proof}

\section{Multiplication operator $T_B$ on the Bergman space as a Bergman Shift}
Let $B$ be a finite Blaschke product of order $n$. We are ready now to show that $T_{B}$ can be represented as a Bergman bundle shift. Recall the set $$\mathcal{S}=
B(\{\beta :B^{\prime}(\beta)=0\}).$$ Note that $\mathcal{S}$ is finite and
$|\mathcal{S}|\leq n-1$ where $|\mathcal{S}|$ denotes the cardinality of $\mathcal{S}$. Let
$\mathcal{S}=\{\beta_1,\ldots,\beta_k\}$ and $l_i$ be the line obtained by
joining $\beta_i$ to boundary of $\mathbb{D}$ with the assumption
that no two lines intersect each other. Set
$\mathbb{D}_B=\mathbb{D}\setminus \cup_{i=1}^{k}l_i$, clearly
$\mathbb{D}_B$ is a simply connected domain. For an open set
$U\subset \mathbb{D}$, we define a \emph{inverse} of $B$ in $U$ to be a
holomorphic function $f$ in $U$ with $f(U)\subset\mathbb{D}$ such
that $B(f(z))=z$ for every $z$ in $U$.

For each $z\in \mathbb{D}\setminus\mathcal{S}$,
$B^{-1}(z)=\{z_1,z_2,\ldots,z_n\}$ with $z_i\neq z_j$ for $i\neq
j$ and $B$ is one to one in some open neighborhood $U_{z_i}$ of
each point $z_i$. Let $U$ be an open neighborhood of $z$ in
$\mathbb{D}_B$ and $\sigma_i$ be the biholomorphic map from $U$ to
$U_{z_i}$, for $1\leq i\leq n$, such that $B(\sigma_i(z))=z$ for
$z\in U$. Since $\mathbb{D}_{B}$ is simply connected so by
Monodromy theorem we can extend holomorphically each $\sigma_i$ to
the domain $\mathbb{D}_{B}$ and we denote the extended holomorphic
function by the same symbol $\sigma_i$.

 Note that here we use inverses of $B$, not local inverses as in \cite{dpw},\cite{dsz}

\begin{lem}\label{lem1}
The images of the $\sigma_i$'s are disjoint, that is,
$\sigma_i(\mathbb{D}_B)\cap\sigma_j(\mathbb{D}_B)=\emptyset$ for
$i\neq j$.
\end{lem}
\begin{proof}
Suppose $w_0\in \sigma_i(\mathbb{D}_B)\cap\sigma_j(\mathbb{D}_B)$,
so there exist $z_0\in\mathbb{D}_B$ such that
$w_0=\sigma_i(z_0)=\sigma_j(z_0)$. Set $\sigma=
\sigma_j^{-1}\circ\sigma_i :\mathbb{D}_B\to\mathbb{D}_B$, here
$\sigma$ is a biholomorphic map (because $\sigma_i'$s are
biholomorphic). Since $\mathbb{D}_B$ is simply connected so by the
Riemann mapping theorem there exists a biholomorphic map $\psi$
from $\mathbb{D}_B$ to $\mathbb{D}$. Set $f= \psi\circ \sigma
\circ \psi^{-1}$, clearly $f$ is a biholomorphic map of
$\mathbb{D}$ and $f(\psi(z_0))=\psi(z_0)$. By Schwarz lemma we
have$$f(z)=
\frac{\psi(z_0)-e^{i\theta}\phi_{\psi(z_0)}(z)}{1-e^{i\theta}\overline{\psi(z_0)}\phi_{\psi(z_0)}(z)}$$
where
$\phi_{\psi(z_0)}(z)=\frac{\psi(z_0)-z}{1-\overline{\psi(z_0)}\,\,z}$
and $\theta\in [0,2\pi)$. Thus $\sigma= \psi^{-1}\circ f\circ
\psi$ which implies that $\sigma_i= \sigma_j\circ\psi^{-1}\circ
f\circ \psi$. Since $\sigma_i$'s are inverses of $B$, so we have
$\psi^{-1}\circ f\circ \psi(z)=z$ for all $z\in \mathbb{D}_B$.
Thus $\sigma_i=\sigma_j$ which is a contradiction. Hence
$\sigma_i(\mathbb{D}_B)\cap\sigma_j(\mathbb{D}_B)=\emptyset$ for
$i\neq j$.
\end{proof}

\begin{lem}\label{lem2}
Let $f$ be a holomorphic function defined on a domain $\Omega$
such that $f(z_0)=w_o$, $f^{(m-1)}(z_0)=0$ and
$f^{(m)}(z_0)\neq0$. The behavior of $f$ near $z_0$ behave like the function $z^m$ near $0$.
\end{lem}
\begin{proof}
Since $f(z)-w_0$ has zero of order $m$ at $z_0$, we have
$f(z)=(z-z_0)^mh(z)$, where $h$ is an analytic function near $z_0$
and $h(z_0)\neq 0$, $h^{\prime}(z_0)\neq 0$. Let  $g(z)=
(z-z_0)h(z)^{1/m}$ then clearly $g$ is holomorphic near $z_0$,
$g^{\prime}(z_0)\neq 0$ and hence g is one to one near $z_0$. Thus
$f(z)=g(z)^m+w_0$ is a composition of three functions, the
function $g(z)$, followed by the map $\zeta\mapsto\zeta^m$,
followed by translations $\xi\mapsto \xi+w_0$.

At point $w$ near $w_0$, $w\neq w_0$, has $m$ distinct preimages
$z_1(w),\ldots,z_m(w)$. They are the $m$ branches of
$(w-w_0)^{1/m}$ composed with $g^{-1}(\zeta)$. If we make branch
cut and consider the principal branch $(w-w_0)^{1/m}$ on the slit
disk $\{|w-w_0|<\delta\}\setminus (w_0-\delta,w_0]$, for some
$\delta>0$, the other branches are of the form $e^{2\pi i
j/m}(w-w_0)^{1/m}$, and preimages of $w$ are given by the
composition $g^{-1}(\zeta)$ and these branches,
$$z_j(w)=g^{-1}(e^{2\pi i
j/m}(w-w_0)^{1/m}),\;\;\;\;\;\;1\leq j\leq m.$$
\end{proof}
A detailed proof of Lemma \ref{lem2} can be found in \cite{gam}.

Let $w_0\in \mathbb{D}\setminus\mathcal{S}$ and $U$ be an open set in
$\mathbb{D}_B$ containing $w_0$. Let $\{\sigma_i\}_{i=1}^{n}$ be
inverses of $B$ defined on $U$. Let $\gamma_i$ be a closed curve
in $\mathbb{D}\setminus\mathcal{S}$ at $w_0$ which inclose the point
$\beta_i$ for $1\leq i\leq k$. If we move $\sigma_j$'s along
closed curve $\gamma_i$ by analytic continuation we get a
permutation on elements $\{1,\ldots,n\}$, say $\tau_i$, which
defines an unitary operator $V_i$ on $\mathbb{C}^n$ as follows:
$$V_i(x_1,\ldots,x_n)=(x_{\tau_i(1)},\ldots,x_{\tau_i(n)})\;\;\;\;\;\;\;  1\leq i\leq k.$$

Let $\alpha:\pi_1(\mathbb{D}\setminus\mathcal{S})\to
\mathcal{U}({\mathbb{C}^n}) $ be a representation. Since
$\pi_1(\mathbb{D}\setminus\mathcal{S})$ is a free group on $k$ generators,
the representation determined by $k$ unitary operators
$V_1,\ldots, V_k$. The representation $\alpha$ defines a flat unitary
vector bundle $E_{B}$ over $\mathbb{D}\setminus\mathcal{S}$. Informally, we
define a topology on $\mathbb{D} \times \mathbb{C}^n$ so that as a
point $(z,v)$ moves continuously across the cut $l_i$ the vector
$v$ becomes $V_i(v)$. Now we are ready to identify $T_{B}$ as a Bergman bundle shift which is the main result of the paper. Not many arguments come from Sun, Zheng, and the first author's paper \cite{dsz}, especially the calculation in the following proof.

\begin{thm} Let $B$ be a finite Blaschke product of order $n$. Let
$T_B$ be the multiplication operator on $L_a^2(\mathbb{D})$ by
$B$. Let $E_{B}$ be the flat unitary vector bundle over
$\mathbb{D}\setminus\mathcal{S}$ determined by $B$. Let $T_{E_{B}}$ be the
multiplication operator on $L_a^2(E_{B})$ by $z$. Then operator
$T_B$ is unitarily equivalent to operator $T_{E_{B}}$.
\end{thm}
\begin{proof} Define a map $\Gamma:L_a^2(\mathbb{D})\to
L_a^2(E_{\alpha})$ as follows $$ \Gamma(f)=
\frac{1}{\sqrt{n}}\big[\big((f\circ
\sigma_1)\sigma_1^{\prime},\ldots, (f\circ
\sigma_n)\sigma_n^{\prime}\big)^{\rm tr}\big]$$ where
$\big[\big((f\circ \sigma_1)\sigma_1^{\prime},\ldots, (f\circ
\sigma_n)\sigma_n^{\prime}\big)^{\rm tr}\big]$ denotes the
equivalence class of $\big((f\circ
\sigma_1)\sigma_1^{\prime},\ldots, (f\circ
\sigma_n)\sigma_n^{\prime}\big)^{\rm tr}$. We show that $\Gamma$
is a unitary map: \\
(1) $\Gamma$ preserves the inner product: \begin{eqnarray*} \langle
\Gamma(f),\Gamma(g)\rangle &=&
\frac{1}{n}\left\langle\big[\big((f\circ
\sigma_1)\sigma_1^{\prime},\ldots, (f\circ
\sigma_n)\sigma_n^{\prime}\big)^{\rm tr}\big],\big[\big((g\circ
\sigma_1)\sigma_1^{\prime},\ldots, (g\circ
\sigma_n)\sigma_n^{\prime}\big)^{\rm tr}\big]\right\rangle\\
&=&-\frac{1}{2n\pi i}\sum_{i=1}^{n}\int_{\sigma_i(\mathbb{D}_B)}
(f\circ\sigma_i)(z)\overline{(g\circ\sigma_i)(z)}
\;\;|\sigma^{\prime}_i(z)|^2
dz\wedge d\bar z\\
&=&-\frac{1}{2n\pi i}\sum_{i=1}^{n}\int_{\mathbb{D}_B}
f(z)\overline{g(z)} \;dz\wedge d\bar z\\
&=&-\frac{1}{2\pi i}\int_{\mathbb{D}_B}
f(z)\overline{g(z)} \;dz\wedge d\bar z\\
&=&-\frac{1}{2\pi i}\int_{\mathbb{D}} f(z)\overline{g(z)}
\;dz\wedge d\bar z\\
&=&\langle f,g \rangle.
\end{eqnarray*}
(2) $\Gamma$ is onto:

Let $[(g_1,\ldots,g_n)^{\rm tr}]$ be an element of $L_a^2(E_B)$.
Define a function $g$ on $\mathbb{D}$ as follows
\begin{eqnarray*}
g(z)=\begin{cases} \sqrt{n}\;g_i\circ B (z)\cdot
((\sigma_i)^{\prime}(B(z)))^{-1}
\;\;\mbox{if}\;\; z\in \sigma_i(\mathbb{D}_B),\;\; 1\leq i\leq n\\
\mbox{defined by equivalence relation}\;\; \mbox{if}\;\; z\in
B^{-1}(\cup_{j=1}^{k}l_j).
\end{cases}
\end{eqnarray*}
Since the values match up on the boundary by construction, $g$ is a holomorphic function on $\mathbb{D}$ and
\begin{eqnarray*}
\Gamma(g)&=&\big[\big((g\circ\sigma_1)\cdot\sigma_1^{\prime},\ldots,(g\circ\sigma_n)\cdot\sigma_n^{\prime}\big)^{\rm tr}\big]\\
&=&\big[\big(g_1\cdot
((\sigma_1)^{\prime})^{-1}\sigma_1^{\prime},\ldots,g_n\cdot
((\sigma_n)^{\prime})^{-1}\sigma_n^{\prime}\big)^{\rm tr}\big]\\
&=&\big[\big(g_1,\ldots,g_n\big)^{\rm tr}\big]
\end{eqnarray*}
\noindent (3) $\Gamma$ intertwines $T_B$ and $T_{E_B}$: For $f\in
L_a^2(\mathbb{D})$
\begin{eqnarray*}
\Gamma\circ T_B(f)&=& \Gamma(Bf)\\
&=&
\frac{1}{\sqrt{n}}\big[\big((B\circ\sigma_1)(f\circ\sigma_1)\cdot\sigma_1^{\prime},\ldots,
(B\circ\sigma_n)(f\circ\sigma_n)\cdot\sigma_n^{\prime}\big)^{\rm
tr}\big]\\
&=&\frac{1}{\sqrt{n}}\big[\big(z(f\circ\sigma_1)\cdot\sigma_1^{\prime},\ldots,
z(f\circ\sigma_n)\cdot\sigma_n^{\prime}\big)^{\rm
tr}\big]\\
&=&z\frac{1}{\sqrt{n}}\big[\big((f\circ\sigma_1)\cdot\sigma_1^{\prime},\ldots,
(f\circ\sigma_n)\cdot\sigma_n^{\prime}\big)^{\rm
tr}\big]\\
&=&T_{E_B}\circ\Gamma(f).
\end{eqnarray*}
Thus $$\Gamma\circ T_B=T_{E_B}\circ\Gamma.$$
\end{proof}

Some arguments in the proof of this result is closely related to the paper \cite{dsz}.

\begin{cor}
$(\mathcal{W}^{\ast}(T_{B}))'$ is isomorphic the commutant algebra of unitary matrixes which define the the bundle.
\end{cor}

\begin{cor}
The restriction of $T_{B}$ to a reducing subspace is a Bergman bundle shift also, which is corresponding to the reduction of the unitaries on that subbundle determined by that subspace.
\end{cor}

In \cite{ron}, model spaces for the special case $B=z^{n}$ is discussed explicitly.

As we mentioned earlier in this section, there is no way to put a canonical order on the set $B^{-1}(w)$ for $w\in\mathbb{D}\setminus\mathcal{S}$. Different choice of lines will yield different orders unless $\mathcal{S}$ has only one element. As a consequence, the representation of the covering group by permutation group is not unique.

Representing $T_{B}$ as a bundle shift allows us to recover most of the results in \cite{dpw},\cite{dsz} except for two key ones: the fact that $(\mathcal{W}^{\ast}(T_{B}))'$ is abelian and its linear dimension.
A more careful analysis of the covering group associated to the Riemann surface $\{(z_{1},z_{2}):B(z_{1})=B(z_{2})\}$ for $B$ will be required for that. Also the latter is a rational function in $z_{1}$ and $z_{2}$, its polynomial numerator is closely related to the permutation covering group, most likely via Galois theory. We leave such analyses to a later paper.

\noindent \textbf{Acknowledgement:} The second and third authors
thank Department of Mathematics, Texas A \& M, University for warm
hospitality. The work of D.K.K was supported partially in the form of a fellowship by
Indo-US The Virtual Institute for Mathematics and Statistical
Sciences (VI-MSS)  and  partially in the form of Stipend by
Department of Mathematics, Texas A\&M, University. The work of A.X
was supported by China Scholarship council and Chongqing University of Technology.


\bibliographystyle{amsplain}
\bibliography{bibliography}
\end{document}